\documentclass[11pt]{amsart}
\usepackage[margin=4 cm]{geometry}
\usepackage{graphicx}
\usepackage{amsmath}
\usepackage{amssymb}
\usepackage{amsthm}
\usepackage{enumerate}
\usepackage{xcolor}
\usepackage[colorlinks]{hyperref}

\newtheorem{theorem}{Theorem}[section]

\newtheorem{lemma}[theorem]{Lemma}
\newtheorem{proposition}[theorem]{Proposition}
\theoremstyle{definition}
\newtheorem{definition}[theorem]{Definition}
\theoremstyle{remark}
\newtheorem{rem}[theorem]{Remark}
\theoremstyle{definition}
\newtheorem{example}[theorem]{Example}
\numberwithin{equation}{section}

\newcommand{\set}[1]{\left\{#1\right\}}

\newcommand{\R}{\mathbb R}
\newcommand{\Z}{\mathbb Z}
\newcommand{\N}{\mathbb N}

\newcommand{\C}{\mathbb C}
\newcommand{\eps}{\varepsilon}
\newcommand{\PP}{\mathcal{P}}
\newcommand{\LL}{\mathcal{L}}
\newcommand{\BB}{\mathcal{B}}
\newcommand{\QQ}{\mathcal{Q}}
\newcommand{\MM}{\mathcal{M}}

\newcommand{\KK}{\mathcal{K}}
\newcommand{\RR}{\mathcal{R}}

\newcommand{\dynzeta}[1]{\zeta_{_{#1}}}
\newcommand{\flattr}{\text{tr}^b}
\newcommand{\flatdet}{\text{det}^b}
\newcommand{\real}[1]{\text{Re($#1$)}}

\begin{document}

\title[]{Dynamical zeta functions for differentiable parabolic maps of the interval}
\author{Claudio Bonanno}
\address{Dipartimento di Matematica, Universit\`a di Pisa, Largo Bruno Pontecorvo 5, 56127 Pisa, Italy}
\email{claudio.bonanno@unipi.it}

\author{Roberto Castorrini}
\address{Scuola Normale Superiore, Piazza dei Cavalieri 7, 56126 Pisa, Italy}
\email{roberto.castorrini@sns.it}

\begin{abstract}
This paper explores the domain of meromorphic extension for the dynamical zeta function associated to a class of one-dimensional differentiable parabolic maps featuring an indifferent fixed point. We establish the connection between this domain and the spectrum of the weighted transfer operators of the induced map. Furthermore, we discuss scenarios where meromorphic extensions occur beyond the confines of the natural disc of convergence of the dynamical zeta function.
\end{abstract}
\maketitle

\section{Introduction} \label{sec:intro}


In the last few years, the \emph{dynamical zeta function} introduced by D. Ruelle in the late 1970s has proved to play a crucial role in the study of the statistical properties of hyperbolic dynamical systems, providing a powerful tool for understanding the underlying structure and behavior of such systems. Its definition is based on the distribution of the periodic orbits of the system, therefore the dynamical zeta functions can be considered a bridge between the microscopic and macroscopic aspects of the system.

One of the most interesting properties of the dynamical zeta function is that its poles are related with the eigenvalues of the associated transfer operator, another important tool in the thermodynamic formalism approach to dynamical systems.
The eigenvalues of the transfer operator, in turn, shed light on the stability and mixing properties of the system, offering a deeper understanding of its overall dynamics. The basic idea is that the smaller is the essential spectrum of the transfer operator, the broader is the domain over which the dynamical zeta function can be meromorphically extended. We refer the reader to \cite{bal-book-2} for a detailed discussion on the subject and for an exhaustive list of the literature on the first steps of the development of this theory. 

It is clear from the basic idea sketched above that it is important to find a good space of functions on which the essential spectral radius of the transfer operator becomes smaller. This is typically obtained by looking at functions of increasing regularity\footnote{In the presence of discontinuities, there exist examples for which this can not be done, as shown in \cite{BuCaJa} and \cite{BuCaCa}. Nevertheless, estimates on the essential spectral radius close to optimal do exist also in this case, as shown in \cite{BaCa}.} if the dynamical system is uniformly hyperbolic. However, in not uniformly hyperbolic systems, capturing information on the dynamical zeta function becomes more challenging due to a lack of information in the spectrum of the associated transfer operator, for example the strategy of looking at the space of functions with high regularity does not work in this situation. An important case for which the classic approach fails is that of parabolic maps, i.e. maps with parabolic fixed points. The properties of the dynamical zeta function have been studied in a symbolic dynamic context equivalent to this case in \cite{Is2}, but the direct approach to differentiable parabolic maps is to our knowledge still lacking.

A notable exception is the case of parabolic maps which have an analytic extension to a complex domain. This case has been studied in \cite{Ru} for maps of the interval and in \cite{PY} for higher dimensional maps. In the analytic case the transfer operator of the map, or of its induction, has strong properties which allow for a complete understanding of the problem in the one-dimensional case. This technique has been used also in \cite{BI} for the Farey map and its relations with the Selberg zeta function on the modular surface. In the multidimensional case the results in \cite{PY} give only a partial meromorphic extension of the dynamical zeta function by using the powerful property of nuclearity for the transfer operator of the induced map.

In this paper we deal with the case of differentiable parabolic maps of the unit interval, under the simplifying assumption that the maps are piecewise monotone with full branches (see Section \ref{sec:setting} for all the assumptions). By using induction and the thermodynamic formalism for uniformly hyperbolic systems, we first show that results similar to those in \cite{Is2} and \cite{PY} hold in this case. Then we introduce a method to further extend the dynamical zeta function under an additional assumption on the potential. Our conjecture is that this method can be improved to obtain a complete understanding of the extension of the dynamical zeta function also in the differentiable context.


The structure of the paper unfolds as follows: Section \ref{sec:setting} introduces the maps $T$ under examination and the key elements, namely, the transfer operator and the dynamical zeta function associated with $T$. The primary outcome, contained in Theorem \ref{thm:zeta-final-1}, stands as the pivotal result of the paper. It encompasses the meromorphic extension of the dynamical zeta function, establishing a connection between its poles and the eigenvalues of the weighted transfer operator of the induced map. Additionally, Theorem \ref{teo:cont-v0} delves into the potential extension of the dynamical zeta function's domain beyond the unit disc defined by the transfer operator associated to the induced map.

Section \ref{sec:spectrum} enumerates and proves the principal characteristics of the spectra of the transfer operators studied. These features are subsequently employed in Sections \ref{sec:zeta-t} and \ref{sec:meromorphic} in order to prove the main theorems, respectively.

\section{Settings and main results} \label{sec:setting}

We consider maps $T:[0,1]\to [0,1]$ satisfying the following assumptions:
\begin{itemize}
\item[(H1)] $T$ is \emph{piecewise monotone} with respect to a partition $J =\set{J_0, J_1}$ of the interval with $J_0=(0,a)$ and $J_1=(a,1)$ for some $a\in (0,1)$;
\item[(H2)] $T$ has \emph{full branches}, that is $\overline{T(J_0)}=\overline{T(J_1)} =[0,1]$;
\item[(H3)] For $r\in \N\cup \{\infty\}$, $r\ge 2$, the restriction $T_0:=T|_{J_0}$ extends to a $C^r$ function on $(0,a]$ and $T_1:=T|_{J_1}$ extends to a $C^r$ function on $[a,1]$. We denote by $\psi_i := T_i^{-1} : (0,1) \to J_i$ the local inverses of $T$;
\item[(H4)] $0$ is an \emph{indifferent fixed point}, that is $T(0)=0$ and $T'(0)=1$ with $T'(x) -1 \sim c\, x^{\alpha}$ as $x\to 0^+$ for some $c>0$ and $\alpha > 0$ (notice that $T_0$ is extendible to a $C^r$ function at $0$ only for $\alpha \ge r-1$);
\item[(H5)] There exist $\eps>0$ and $\rho>1$ such that $T'(x)$ is increasing in $(0,\eps)$ and $|T'(x)|\ge \rho$ for all $x\in [\eps, a) \cup (a,1)$.
\end{itemize}

\begin{example}
A classical example of a family of maps satisfying our assumptions are the so-called \emph{Pomeau-Manneville maps} $T_{PM}:[0,1]\to [0,1]$, introduced and studied in \cite{M,PM} to investigate the phenomenon of intermittency in physics and defined as
\[
T_{PM}(x) := x + x^{1+\alpha} \, \text{(mod $1$),} \quad \text{for $\alpha>0$.}
\]
A somewhat simpler family of maps with the same dynamical properties have been introduced in \cite{LSV} and are now referred to as \emph{Liverani-Saussol-Vaienti (LSV) maps}. The maps have been defined as
\[
T_{LSV}(x):=\begin{cases}
x(1+2^\alpha\, x^{\alpha}), & \text{for $x\in [0, \frac 12)$,}\\[0.1cm]
2x-1, & \text{for $x\in [\frac 12, 1]$},
\end{cases}
\]
for $\alpha\in (0,1)$, and clearly can be considered for all $\alpha>0$. These families of maps, together with generalizations, are studied in \cite{BL} in the case $\alpha\ge 1$ as models of infinite measure-preserving interval maps.
\end{example}

The main objects we investigate in the paper are the transfer operator and the dynamical zeta function associated to the map $T$.
Let $v:[0,1] \to \R$ be a function satisfying:
\begin{itemize}
\item[(H6)] for some $k\in \N \cup \{\infty\}$, $k< r$, $v$ is in $C^k(J_i)$ for $i=0,1$ and extends to a $C^k$ function on $[0,a]$ and $[a,1]$.
\end{itemize}
Then we consider the \emph{transfer operator $\LL_v : L^1(0,1)\to L^1(0,1)$ with potential $v$} given by 
\[
(\LL_v f) (x) := e^{v(\psi_0(x))}\, f(\psi_0(x)) + e^{v(\psi_1(x))}\, f(\psi_1(x)).
\]
In the following it is useful to write $\LL_v f$ as the sum of two terms, $\LL_{0,v} f$ and $\LL_{1,v} f$, which represent the action of the \emph{parabolic} and \emph{expanding} part of the map respectively:
\begin{equation}\label{notat-L0L1}
\LL_{i,v} f := (e^v\cdot f) \circ \psi_i\, , \qquad i=0,1\, .
\end{equation}
An important case is given by the choice $v = - \log |T'|$, for which $\LL_v$ reduces to be the \emph{Perron-Frobenius operator}, which is associated to the existence of an absolutely continuous invariant measure. It is known that such invariant measure is finite if $\alpha\in (0,1)$ and infinite for $\alpha \ge 1$. In the so-called thermodynamic formalism approach, it is also interesting to consider the more general choice $v = -q\, \log |T'|$ with $q\in [0,+\infty)$ (and sometimes $q\in \C$, see \cite{BI}). \\

By now, a standard way to study parabolic maps of the interval is inducing on the sets where the dynamics is expanding. Let
\begin{equation}\label{tau-function}
\tau(x) := \min \set{n\ge 0\, :\, T^n(x) \in J_1}
\end{equation}
be the \emph{first passage function} from $J_1$. It is finite on $[0,1]$ but for a countable set. Let 
\[
a_0:=1,\quad a_1:= a,\quad a_\ell := \psi_0(a_{\ell-1})\quad \text{for $\ell\ge 2$,}
\]
then $\tau(x) <\infty$ if and only if $x\not=0$ and $x \not= a_\ell$ for all $\ell\ge 0$. In addition we introduce the notation
\[
A_\ell := (a_\ell,\, a_{\ell-1}) = \set{x\in (0,1)\, :\, \tau(x) = \ell-1}\quad \text{for $\ell\ge 1$,}
\]
for the level sets of $\tau$. It is well known (see e.g. \cite[Lemma 2.1]{Is0}) that there exists a positive constant only depending on $T$ such that\footnote{The constants involved in the inequalities are written in the form $const(\cdot)$ to specify their dependence and they might be different from one line to the next one even if written with the same notation.}
\begin{equation}\label{asintotica-al}
a_\ell \sim \, const(T)\ \ell^{-1/\alpha}\, (1+O(\ell^{-1}))\quad \text{as $\ell \to \infty$.}
\end{equation}

We are now ready to introduce the \emph{jump transformation} of $T$ on $J_1$ which is the map
\[
G: \bigcup_{\ell\ge 1}\, A_\ell \to (0,1),\quad G(x) := T^{1+\tau(x)}(x).
\]
It is immediate to verify that $G$ has the following properties.
\begin{lemma}\label{lem:prop-jt}
Let (H1)-(H5) hold for $T$, then for any $\ell \ge 1$ the restriction $G_\ell := G|_{A_\ell}$ satisfies:
\begin{enumerate}[(i)]
\item $G_\ell = T_1 \circ T_0^{\ell-1}$;
\item $G_\ell(A_\ell) = (0,1)$;
\item It is monotone with inverse $\phi_\ell : (0,1) \to A_\ell$ given by $\phi_\ell = \psi_0^{\ell-1} \circ \psi_1$;
\item It extends to a $C^r$ function on $\overline{A_\ell} = [a_\ell,\, a_{\ell-1}]$ and $\phi_\ell$ extends to a $C^r$ function on $[0,1]$;
\item $|G'_\ell|\ge \rho$.
\end{enumerate}
\end{lemma}

\begin{example}\label{ex:farey-gauss}
A classical example of a map which can be obtained as a jump transformation is the \emph{Gauss map} $G(x)= 1/x - \lfloor 1/x \rfloor$ for $x\in (0,1]$. Let
\begin{equation}\label{farey}
T_F(x) := \begin{cases}
\frac{x}{1-x}, & \text{for $x\in [0, \frac 12]$,}\\[0.1cm]
\frac{1-x}{x}, & \text{for $x\in [\frac 12, 1]$},
\end{cases}
\end{equation}
be the \emph{Farey map}, then the Gauss map is the jump transformation of $T_F$ on $(1/2,1)$. The Farey map satisfies (H1)-(H5)\footnote{{ Since $|T'_F(1)|=1$, (H5) is satisfied only away from 1. This is not a problem, since to obtain the condition in Lemma \ref{lem:prop-jt}-(v) for the Gauss map, it is enough to look at the second iterate $G^2$.}} with $\alpha=1$ and $r=\infty$, but can also be extended to an analytic function on a strip containing the interval $(0,1)$. Hence the dynamical zeta function of $T_F$ has been studied by the method of Fredholm determinants for analytic maps (see \cite{Ru} and \cite{BI}).
\end{example}


The next step is the application of the inducing procedure to define the transfer operator for $G$. Given a potential function $v:(0,1) \to \R$ satisfying (H6), we define the \emph{induced potential} to be the function
\begin{equation}\label{induced-pot}
w: \bigcup_{\ell\ge 1}\, A_\ell \to \R,\quad w(x) := \sum_{j=0}^{\tau(x)}\, v(T^j(x)).
\end{equation}
For any $\ell\ge 1$, the function $w$ extends to a $C^k$ function on $\overline{A_\ell}$. Then we introduce the \emph{transfer operator-valued power series $\QQ_w(z)$ with potential $w$ and variable $z\in \C$} as the formal series acting on a function $f\in L^1(0,1)$ as
\begin{equation}\label{trans-op-G}
(\QQ_w(z)f) (x) := \sum_{n=1}^\infty\, z^n\, e^{w(\phi_n(x))}\, f(\phi_n(x)). 
\end{equation}
Notice that the choice $v= -q \log |T'|$ leads to $w= -q \log |G'|$, and for $q=1$ the operator $\QQ_w(1)$ is the Perron-Frobenius operator of the map $G$. 

In order to obtain good convergence properties for $\QQ_w(z)$ we assume that:
\begin{itemize}
\item[(H7)] there is a constant $c_k>0$ such that for all $h\in \N \cup \{\infty\}$, $h\le k$,
\[
\sum_{n=1}^\infty\, \| e^{w\circ \phi_n} \|_h\, \Big( 1+ \| \phi_n \|_h^h \Big) \le c_k,
\]
\end{itemize}
where for $f\in C^k([0,1])$ we denote 
\[
\|f\|_h := \max_{0\le j\le h} \| f^{(j)}\|_\infty,
\]
 the max of the sup-norm of $f$ and its first $h$ derivatives.

For $w= -q \log |G'|$ and $G$ the Gauss map of Example \ref{ex:farey-gauss} we obtain by \eqref{asintotica-al} that (H7) holds for $\Re(q)> \alpha/(1+\alpha)$.
 
\begin{rem} \label{rem:v0}
By definition, $\phi_n(x) \in A_n$ for all $n\ge 1$ and $T^j(A_n)=A_{n-j}$ for all $0\le j\le n-1$. Hence
\[
w(\phi_n(x)) = \sum_{j=0}^{n-1}\, v(T^j(\phi_n(x)))
\]
satisfies
\[
\sum_{i=1}^{n}\, \min_{x\in A_i}\, v(x) \le w(\phi_n(x)) \le \sum_{i=1}^{n}\, \max_{x\in A_i}\, v(x).
\]
Therefore, it follows $w(\phi_n(x)) \sim n\cdot v(0)$ as $n\to \infty$. Hence (H7) implies $v(0)\le 0$.
\end{rem} 

In the following we need to consider iterates of $\QQ_w(z)$ which involve compositions of the local inverse of $G$ and of the potential $w$. Given the multi-indices $\beta = (\beta_1,\, \beta_2,\dots,\beta_m) \in \N^m$ for which we set $|\beta| := \beta_1+\beta_2+\dots+\beta_m$, we introduce the following notations:
\begin{equation}\label{notations}
\begin{aligned}
&\phi_\beta := \phi_{\beta_1} \circ \phi_{\beta_2} \circ \dots \circ \phi_{\beta_m},\\
&\phi_{_{\beta_i\dots\beta_h}} := \phi_{\beta_i} \circ \phi_{\beta_{i+1}} \circ \dots \circ \phi_{\beta_h}, \quad \forall\, 1\le i \le h\le m,\\
&W_\beta(x):= \exp \Big( \sum_{j=1}^m\, w(\phi_{_{\beta_j\dots\beta_m}}(x)) \Big).
\end{aligned}
\end{equation}

The main results use the following quantities:
\begin{equation}\label{lambda-m}
\Lambda_m(z) := \sup_{x\in (0,1)}\, \sum_{\beta \in \N^m}\, |z|^{|\beta|}\, W_\beta(x),
\end{equation}
and the limit
\begin{equation}\label{lambda-z}
\Lambda(z):= \limsup_{m\to \infty}\, [\Lambda_m(z)]^{\frac 1m}.
\end{equation}
Their existence is studied in Theorem \ref{thm:q-ck} below. \\
 
Finally,  we recall the definition of the \emph{dynamical zeta function} of the map $T$ with potential $v$, which is defined for $z\in \C$ as
\begin{equation}\label{zeta-t}
\dynzeta{T,v}(z) := \exp \Big( \sum_{n\ge 1}\, \frac{z^n}n\, \sum_{T^n(x)=x}\, \exp\Big(\sum_{j=0}^{n-1}\, v(T^j(x))\Big) \Big).
\end{equation}
It is a straightforward computation that the series in \eqref{zeta-t} converges for 
\[
|z| < \exp(-P_T(v)), \quad P_T(v):= \limsup_{n\to \infty}\, \frac 1n\, \log \Big( \sum_{T^n(x)=x}\, \exp\Big(\sum_{j=0}^{n-1}\, v(T^j(x))\Big) \Big),
\]
where $P_T(v)$ is the so-called \emph{topological pressure} of the map $T$ for the potential $v$ and it is finite by (H7). In the case $v=-\log |T'|$, the topological pressure of $T$ vanishes. The behaviour of the function $P_T(-q\, \log |T'|)$ for $q\in [0,\infty)$ is studied in the thermodynamic formalism approach to the map $T$. For example, $P_T(0)$ is the topological entropy of $T$ and the lack of regularity of this function at $q=1$ is an interesting phenomenon called phase transition.

Our main results are a contribution to the important open problem of the extension of $\dynzeta{T,v}(z)$ outside its $z$-disc of convergence. For the proofs see Section \ref{sec:zeta-t} and \ref{sec:meromorphic}.

\begin{theorem}\label{thm:zeta-final-1}
Under assumptions (H1)-(H7), the dynamical zeta function $\dynzeta{T,v}(z)$ is defined on the disc $\{z < \exp(-P_T(v))\}$, {and has a meromorphic extension to the set
\[
\{z\in \C\, :\, |z|<1\, ,\, \Lambda(z)<\rho^{r-2}\}\, ,
\]  
}with singularities at $z_0 = e^{-v(0)}$ and at the values $\bar z$ for which 1 is an isolated eigenvalue of $\QQ_w(\bar z)$, and with zeroes at the values $\tilde z$ for which 1 is an isolated eigenvalue of $\QQ_{w-\log G'}(\tilde z)$ (see \eqref{q-w-log}).
\end{theorem}


Our second result {discusses the possible extension to larger domains. In particular, we show that, for some potentials $v$, the transfer operator-valued function $z\mapsto \QQ_w(z)$ admits an analytic extension to domains bigger than $D^\circ =\{z\in \C: |z|<1\}$, which are unbounded for $r=\infty$.

\begin{theorem}\label{teo:cont-v0}
Under assumptions (H1)-(H7), if $v(0)<0$ the transfer operator-valued function $z\mapsto \QQ_w(z)$ has an analytic extension as an operator on $C^k([0,1])$ to the set 
\[
z\in \{z\in \C\, :\, |z|< 1\} \cup (\C\setminus \bar{\Delta})\, ,
\] 
where
\[
\bar{\Delta}:= \{z\in \C\, :\, |\arg z| \le \max\{0,\pi+v(0)\} \}.
\]
\end{theorem}
When $\pi+v(0)\le 0$ the sector $\bar{\Delta}$ reduces to the line $(0,+\infty)$, so that we obtain the extension to the cut domain $\C\setminus [1,+\infty)$.}

\begin{rem} \label{rem:peccato}
Unfortunately, this method proves ineffective for the case where $v(0)=0$, as seen in the Perron-Frobenius scenario where $v = -\log |T'|$. The problem is to find a suitable function $h$ that satisfies Lemma \ref{lem:extensions} with $e_{_\Pi}<\pi$ in this instance, we believe this is an interesting open question. In general, we conjecture that in all cases the {function $z\mapsto \QQ_w(z)$ has an analytic extension to the set $\C\setminus [1,+\infty)$.}
\end{rem}

{
\begin{rem} \label{rem:zeta-ext}
Let's also assume that the analytic extension of $\QQ_w(z)$ satisfies the same properties used in the proof of Theorem \ref{thm:zeta-final-1}, namely the existence of a spectral gap on $C^k([0,1])$ and the convergence of the flat determinant \eqref{flat-det} on a polydisc, and the same holds for $\QQ_{w-\log G'}(z)$. Then, we obtain that if $v(0)<0$ the dynamical zeta function $\dynzeta{T,v}(z)$ has a meromorphic extension to the set 
\[
\{z\in D^\circ \cup (\C\setminus \bar{\Delta})\, :\, \Lambda(z)<\rho^{r-2} \},
\] 
with singularities at $z_0 = e^{-v(0)}$ and at the values $\bar z$ for which 1 is an isolated eigenvalue of $\QQ_w(\bar z)$, and has zeroes at the values $\tilde z$ for which 1 is an isolated eigenvalue of $\QQ_{w-\log G'}(\tilde z)$. When $\pi+v(0)\le 0$, we obtain the extension to the cut domain $(\C\setminus [1,+\infty)) \cap \{\Lambda(z)<\rho^{r-2}\}$, which is the cut plane $\C\setminus [1,+\infty)$ when $r=\infty$.
\end{rem}
}


\section{The spectrum of the transfer operators} \label{sec:spectrum}
In this section we study the spectral properties of the transfer operators of $T$ and of the corresponding induced map $G$ on the space of differentiable functions. The results for $G$ are expected by the previous literature on the subject. We give some proofs for completeness. 

\begin{theorem} \label{thm:q-ck}
Let (H1)-(H7) hold. Then for all $z \in D := \{|z|\le 1\}$, the power series $\QQ_w(z)$ converges to a bounded operator from $C^k([0,1])$ to $C^k([0,1])$. In addition, recalling \eqref{lambda-m} and \eqref{lambda-z}, we have for all $z\in D$:
\begin{enumerate}[(i)]
\item The limit
\[
\Lambda(z) = \lim_{m\to \infty}\, [\Lambda_m(z)]^{\frac 1m}
\]
exists and $\Lambda(z)\le c_k$;
\item The essential spectral radius of $\QQ_w(z)$ on $C^k([0,1])$ satisfies $r_{ess}(\QQ_w(z)|_{C^k})\le \rho^{-k} \Lambda(z)$;
\item If $z\in (0,1]$ then $\Lambda(z)$ is an eigenvalue of $\QQ_w(z)$ on $C^k([0,1])$.
\end{enumerate}
\end{theorem}

\begin{proof}
The proof of the first assertion goes as follows.
Let us write $\QQ_w(z)$ as 
\[
(\QQ_w(z)f)(x) = \sum_{n=1}^\infty\, z^n\, (Q_{n,w} f)(x),
\]
where
\begin{equation}\label{notat-Qn}
Q_{n,w} f:= (e^w \cdot f) \circ \phi_n, \quad n\ge 1.
\end{equation}
Then we estimate $\|Q_{n,w} f\|_k$ for all $n\ge 1$. We use two formulas to compute the derivatives of the product and composition of functions. One is the well-known Leibniz's formula 
\[
(f\cdot g)^{(j)} = \sum_{i=0}^j\, {j \choose i}\, f^{(i)}\, g^{(j-i)}.
\]
The other is the Fa\`a di Bruno's formula (see e.g. \cite{Faa}). Given a multi-index $a=(a_1,a_2,\dots,a_j) \in \N_0^j$ the formula reads
\[
(f\circ \phi)^{(j)} = \sum_{s=1}^j\, (f^{(s)} \circ \phi)\, \sum_{a \in B_{j,s}}\, \frac{j!}{a_1!\dots a_j!}\, \prod_{\ell=1}^j\, \Big( \frac{1}{\ell!}\, \phi^{(\ell)} \Big)^{a_\ell},
\]
where $B_{j,s} = \{ a\in \N_0^j\, :\, a_1+\dots+a_j=s,\, a_1+2a_2+\dots+ja_j=j\}$. Thus for all $j=1,\dots,k$, 
\[
\| (Q_{n,w} f)^{(j)}\|_\infty \le \sum_{i=0}^j\, {j \choose i}\, \| (e^{w\circ \phi_n})^{(j-i)}\|_\infty \, \| (f\circ \phi_n)^{(i)}\|_\infty
\]
and
\[
\begin{aligned}
\| (f\circ \phi_n)^{(j)}\|_\infty \le\, & \sum_{s=1}^j\, \| f \|_s \, \sum_{a \in B_{j,s}}\, const(j,s)\, \prod_{\ell=1}^j\, \| \phi_n^{(\ell)} \|_\infty^{a_\ell} \\
\le\, & const(j)\, \sum_{s=1}^j\, \| f \|_s\, \| \phi_n\|_j^j \le\, const(j)\, \| f \|_j\, \| \phi_n\|_j^j.
\end{aligned}
\]
Therefore for all $n\ge 1$
\[
\| Q_{n,w} f \|_k \le const(k)\, \| e^{w\circ \phi_n} \|_k \, \| \phi_n\|_k^k\, \| f \|_k.
\]
It follows by (H7) that for all $z\in D$
\[
\| \QQ_w(z)f \|_k \le const(k)\, \sum_{n=1}^\infty\, |z|^n\, \| e^{w\circ \phi_n} \|_k \, \| \phi_n\|_k^k\, \| f \|_k \le const(k)\, \| f \|_k.
\] 
Hence $\QQ_w(z)$ converges absolutely on $D$ to a bounded operator from $C^k([0,1])$ to $C^k([0,1])$.

(i) By definition, for all $z\in D$ and $m\ge 1$,
\[
\begin{aligned}
\Lambda_m(z) \le\, & \Big\| \sum_{\beta_1,..,\beta_{m-1}=1}^\infty\, |z|^{\beta_1+\beta_2+..+\beta_{m-1}}\, \exp\Big(\sum_{\ell=1}^{m-1}\, w\circ \phi_\ell\Big) \Big\|_\infty \, \sum_{\beta_m=1}^\infty\, |z|^{\beta_m}\, \| e^{w\circ \beta_{m}} \|_\infty\\
\le\, & c_k \, \Big\| \sum_{\beta_1,..,\beta_{m-1}=1}^\infty\, |z|^{\beta_1+..+\beta_{m-1}}\, \exp\Big(\sum_{\ell=1}^{m-1}\, w\circ \phi_\ell\Big) \Big\|_\infty = c_k\, \Lambda_{m-1}(z),
\end{aligned}
\]
hence $\Lambda_m(z) \le c_k^m$. In addition, it is immediate to show that $\Lambda_{m+m'}(z) \le \Lambda_m(z)\cdot \Lambda_{m'}(z)$, hence $\Lambda(z)$ exists and satisfies $\Lambda(z) \le c_k$.

(ii) follows as in \cite[Thm. 3.1]{kalle} (see also \cite[Thm. 6]{CI} and \cite[Thm. 2.5]{bal-book-1}). The main idea is to notice that for all $j=1,\dots,k$ one can write as above
\[
(\QQ_w(z)f)^{(j)} = \sum_{n=1}^\infty\, z^n\, e^{w\circ \phi_n}\, (f^{(j)}\circ \phi_n)\, (\phi_n')^j + \RR_{z,j,w} f\, ,
\]
with $\RR_{z,j,w}$ compact as an operator from $C^j([0,1])$ to $C^0([0,1])$. Applying this idea to $(\QQ_w(z))^m$ for all $m\ge 1$ implies that, denoting with $\KK(C^k)$ the set of compact operators from $C^k([0,1])$ to $C^k([0,1])$,
\[
\inf_{\RR \in \KK(C^k)}\, \| (\QQ_w(z))^m - \RR \| \le \sup_{x\in (0,1)}\, \sum_{\beta \in \N^m}\, |z|^{|\beta|}\, W_\beta(x)\, |\phi'_\beta(x)|^k\, .
\]
The result then follows from \eqref{lambda-m}-\eqref{lambda-z} and the classical Nussbaum formula for the essential spectral radius.

(iii) follows as in \cite[Thm. 3.1]{Is2}. 
\end{proof}

Properties (i)-(ii) and (iii) in Theorem \ref{thm:q-ck} tell us that $\QQ_w(z)$ has a spectral gap on $C^k$ for $z\in (0,1]$, which implies exponential decay of correlations for $C^k$ observables for the map $G$ and has implications for the analytic extension of the dynamical determinants we study in the next section. \\

It is well known that we should not expect the operators $\LL_v$ to have a spectral gap on $C^k$ for all $k\in \N \cup \{\infty\}$ or even on the space of holomorphic functions for maps $T$ with an analytic extension. Nevertheless, in the analytic case, it has been proved in \cite{Ru} that the spectrum of the operator $\LL_{0,v}$ on a space of holomorphic functions is the unit interval $[0,1]$. This property has important implications in the definition of a dynamical determinant for the map $T$ in the analytic case. This approach, which we discuss in the next section, has been used also in \cite{Is1,BI} for the Farey map \eqref{farey}. We show that this approach may fail in the differentiable case as the next result holds.

\begin{proposition}\label{prop:l0-spec}
Let $T:[0,1]\to [0,1]$ be a map satisfying (H1)-(H5) with $\alpha\ge 1$, and let $\LL_0$ be the part of the Perron-Frobenius operator of $T$ associated to the parabolic branch, that is $\LL_0 f = \psi'_0 \, (f\circ \psi_0)$. The essential spectral radius of $\LL_0$ on $C^1$ is greater than or equal to 1.
\end{proposition}

\begin{proof}
We show that every $\lambda \in \C$ with $|\lambda|<1$ is an eigenvalue for $\LL_0$ on $C^1$. Let $f\in C^1([0,1])$ and let $F\in C^{2}$ satisfy $F'=f$. Notice that $T$ is $C^2$ at 0 and the same is true for $\psi_0$. Then
\[
\LL_0 f = \lambda \, f \quad \Leftrightarrow \quad \psi'_0 \, (f\circ \psi_0) = \lambda \, f \quad \Leftrightarrow \quad (F\circ \psi_0)' = \lambda\, F'.
\]
Therefore $\lambda \in \C$ is an eigenvalue for $\LL_0$ on $C^1$ if and only if there exists $F\in C^2$ such that
\begin{equation} \label{new-form-l0}
(F\circ \psi_0)(x) = \lambda\, F(x) + const(F),\qquad \forall\, x\in [0,1].
\end{equation}
We show the existence of a solution to \eqref{new-form-l0} for all $\lambda \in \C$ with $|\lambda|<1$.

First of all, since $\psi_0(0)=0$ and $\psi_0'(0)=1$, then a solution $F$ to \eqref{new-form-l0} with $\lambda\not=1$ satisfies $F'(0)=0$, and the same must be true for the second derivative, that is $F''(0)=0$. 

Let now $h:\R \to \R$ be a $C^2$ function with support in $\overline{A_0} =[a,1]$ and such that $h^{(j)}(a) = h^{(j)}(1) =0$ for $j=1,2$. Then for all $\lambda \in \C$ with $|\lambda|<1$
\begin{equation}\label{f-lambda}
F_\lambda(x) := h(x) + \sum_{\ell=1}^{+\infty}\, \lambda^\ell\, \chi_{_{\overline{A_\ell}}}(x)\, h(T^\ell(x)),\qquad \text{for $x\in (0,1]$}
\end{equation}
with $F_\lambda(0)=0$ is a solution to \eqref{new-form-l0} with $const(F)=0$, where $\chi_{A}$ is the indicator function of the set $A$. Since $|\lambda|<1$ we show that the series is absolutely convergent in $C^2$. The function $\chi_{_{\overline{A_\ell}}}(x)\, h(T^\ell(x))$ is in $C^2([0,1])$ for all $\ell\ge 1$, and we can compute
\[
\| \chi_{_{\overline{A_\ell}}}(x)\, h(T^\ell(x)) \|_\infty \le \|h\|_2, \quad \forall\, \ell\ge 1,
\]  
\[
\| \Big( \chi_{_{\overline{A_\ell}}}(x)\, h(T^\ell(x)) \Big)' \|_\infty \le \| h\|_2\, \prod_{i=1}^{\ell}\, \| T' \|_{\infty,A_i},
\]
where $\|\cdot\|_{\infty,A_i}$ denotes the sup-norm of a function restricted to $A_i$, and
\[
\| \Big( \chi_{_{\overline{A_\ell}}}(x)\, h(T^\ell(x)) \Big)'' \|_\infty \le \| h\|_2\, \prod_{i=1}^{\ell}\, \| T' \|_{\infty,A_i}^2\, \Big(1+ \ell \|T\|_2\Big).
\]
By (H5), for $\ell$ big enough so that $a_{\ell-1}<\eps$, we have 
\[
\| T' \|_{\infty,A_\ell} = T'(a_{\ell-1}) \sim 1 + c a_{\ell-1}^\alpha \sim 1 + const(T) \, \ell^{-1} \quad \text{as $\ell \to \infty$,}
\]
using (H4) and \eqref{asintotica-al}, where $const(T)>0$. Hence
\[
\prod_{i=1}^{\ell}\, \| T' \|_{\infty,A_i} \sim \ell^{const(T)} \quad \text{as $\ell \to \infty$.}
\]
This implies that $F_\lambda \in C^2([0,1])$. In addition, using that $\psi_0([0,1]) = [0,a]$, $\psi_0(\overline{A_{\ell-1}}) = \overline{A_\ell}$ and $T^\ell \circ \psi_0(x) = T^{\ell-1}(x)$ for $x\in\overline{A_{\ell-1}}$, we can compute 
\[
\begin{aligned}
F_\lambda(\psi_0(x)) =\, & h(\psi_0(x)) + \sum_{\ell=1}^{+\infty}\, \lambda^\ell\, \chi_{_{\overline{A_\ell}}}(\psi_0(x))\, h(T^\ell(\psi_0(x))) \\
=\, & \sum_{\ell=1}^{+\infty}\, \lambda^\ell\, \chi_{_{\overline{A_{\ell-1}}}}(x)\, h(T^{\ell-1}(x)) \\
=\, & \lambda\, h(x) + \lambda\, \sum_{\ell=1}^{+\infty}\, \lambda^\ell\, \chi_{_{\overline{A_\ell}}}(\psi_0(x)) = \lambda\, F_\lambda(x)
\end{aligned}
\]
for all $x\in (0,1]$. This completes the proof.
\end{proof}

\begin{rem}
In the proof we have used that $0$ is an indifferent fixed point to obtain that the term $\prod_{i=1}^{\ell}\, \| T' \|_{\infty,A_i}$ diverges with polynomial speed in $\ell$. This is false for uniformly expanding maps for which $T'(0)\ge \rho>1$, and in this case the construction of the eigenfunction $F_\lambda$ works only for $|\lambda|<1/\rho$, which is the classical bound on the essential spectral radius of the Perron-Frobenius operator of $T$ on $C^1$. 
\end{rem}

\begin{rem}
We expect this proposition to hold on $C^k$ for $k\ge 2$ under suitable assumptions on the higher derivatives of $T$ which guarantee the total convergence of the series \eqref{f-lambda} in $C^k$ for $|\lambda|<1$. However, this construction does not work in the analytic case since the solution to \eqref{new-form-l0} has vanishing derivatives of all order at 0. Hence, we are describing a property which is typical of the differentiable case.
\end{rem}

\section{The dynamical zeta function of $T$ and $G$} \label{sec:zeta-t}
We now discuss the properties of the dynamical zeta function $\zeta_{T,v}$ defined in \eqref{zeta-t}. It is known that it is a good strategy to use the induced map. Hence we consider the dynamical zeta function of the jump transformation $G$ with the induced potential $w$ (given in \eqref{induced-pot}) for $u\in \C$ given by
\[
\dynzeta{G,w}(u) := \exp \Big( \sum_{m\ge 1}\, \frac{u^m}m\, \sum_{G^m(x)=x}\, \exp\Big(\sum_{j=0}^{m-1}\, w(G^j(x))\Big) \Big),
\]
which converges for $|u|<\exp(-P_G(w))$ where 
\[
P_G(w) := \limsup_{m\to \infty}\, \frac 1m\, \log \Big( \sum_{G^m(x)=x}\, \exp\Big(\sum_{j=0}^{m-1}\, w(G^j(x))\Big) \Big)
\]
is the topological pressure of $G$ for the potential $w$.

The relation between the map $T$ and its jump transformation $G$ reflects into well-known relations between the associated transfer operators and dynamical zeta functions. From \eqref{notat-Qn}, using \eqref{induced-pot} with $\tau(x)=n-1$ and Lemma \ref{lem:prop-jt}-(iii) we write
\[
Q_{n,w}f = \LL_{1,v} \, \LL_{0,v}^{n-1} f,
\]
see \eqref{notat-L0L1}, hence the power series $\QQ_w(z)$ in \eqref{trans-op-G} can be formally written as
\[
\QQ_w(z) f = \sum_{n=1}^\infty\, z^n\, \LL_{1,v} \, \LL_{0,v}^{n-1} f = z\, \LL_{1,v} \Big( \sum_{n=1}^\infty\, (z\, \LL_{0,v})^{n-1} f \Big).
\]
A straighforward manipulation implies that
\begin{equation}\label{rel-trans-op-tg}
\Big(1-\QQ_w(z)\Big) \, \Big(1-z\, \LL_{0,v}\Big) = 1-z\, \LL_v.
\end{equation}
Proving that \eqref{rel-trans-op-tg} holds on a suitable space of functions gives the relation between the spectral properties of the two operators, in particular relating functions fixed by $\QQ_w(z)$ with eigenfunctions of $\LL_v$. This approach has been used in \cite{Is1,BI} for the Farey and the Gauss map. It is precisely because of \eqref{rel-trans-op-tg} that it is interesting to study the spectrum of $ \LL_{0,v}$: If the spectrum of $ \LL_{0,v}$ is the unit interval $[0,1]$ we can write $(1-\QQ_w(z))$ as $(1-z\, \LL_v)(1-z\, \LL_{0,v})^{-1}$ for all $z$ in the cut complex plane $\C\setminus [1,\infty)$, obtaining the analytic extension of the power series $\QQ_w(z)$ to $\C\setminus [1,\infty)$. However, by Proposition \ref{prop:l0-spec} this approach does not work for the maps $T$ we are studying in this paper.

From \eqref{rel-trans-op-tg}, it is possible to get the intuition that a relation should hold also between $\dynzeta{T,v}$ and $\dynzeta{G,w}$. The immediate idea is to write the zeta function $\dynzeta{T,v}$ as the determinant, in some sense, of $(1-z\, \LL_v)$ and to do the same for the jump transformation. Building on this idea, one obtains a relation between the two dynamical zeta functions by using the relation between the periodic points of $T$ and $G$. If $G^m(x)=x$ then $T^n(x)=x$ where $n=\sum_{j=0}^{m-1} (1+\tau(G^j(x)))$. Then, following \cite{PY,Is1} we introduce the \emph{two-variable dynamical zeta function} $Z_w(z,u)$ defined for $z,u \in \C$ as
\begin{equation}\label{2var-zeta}
Z_w(z,u) := \exp \Big( \sum_{m\ge 1}\, \frac{u^m}m\, \sum_{G^m(x)=x}\, z^{\sum_{j=0}^{m-1} (1+\tau(G^j(x)))}\, \exp\Big(\sum_{j=0}^{m-1}\, w(G^j(x))\Big) \Big).
\end{equation}
The function $Z_w(z,u)$ is well-defined for $z\in D= \{ |z| \le 1\}$ and $|u|< \exp(-P_G(u))$. It is immediate to get for $z\not= 0$
\begin{equation}\label{2var-zeta-relations}
\dynzeta{G,w}(u) = Z_w(1,u) \quad \text{and} \quad Z_w(z,u) = \dynzeta{G,w+(\log z)(1+\tau)}(u),
\end{equation} 
and the relation with the dynamical zeta function of $T$ is given by the next result.
\begin{proposition}[\cite{PY,Is1}] \label{prop:standard}
We can write
\begin{equation}\label{rel-zeta-t-fund}
\dynzeta{T,v}(z) = (1-z\, e^{v(0)})^{-1}\, Z_w(z,1)
\end{equation}
whenever the three terms make sense.
\end{proposition}
In particular, $(1-z\, e^{v(0)})^{-1}$ is a meromorphic function on $\C$ and it remains to study the meromorphic extensions of {$z\mapsto Z_w(z,1)$, where we recall that $w$ is the fixed induced potential \eqref{induced-pot}}. To this aim we recall the notions of \emph{flat trace} and \emph{flat determinant} of a linear operator. See \cite[Section 3.2.2]{bal-book-2} for more details and the proofs of the properties we recall. Let $\gamma : [0,1] \to [0,\infty)$ be a $C^\infty$ compactly supported function with $\int_0^1 \gamma(x) dx =1$. Setting $\gamma_\eps(x) := \eps^{-1} \gamma(x/\eps)$ for small $\eps>0$, one introduces the \emph{mollifier operator} $I_\eps$ associated to $\gamma$ on $L^1(0,1)$ as
\[
\Big(I_\eps(f)\Big)(x) := \int_0^1\, \gamma_\eps(x-y)\, f(y)\, dy.
\]
Basic properties of $I_\eps$ include that $I_\eps(f) \in C^\infty$ for all $f\in L^1$ and there exists a kernel $K_\eps(x,y) \in C^\infty([0,1]\times [0,1])$ with support converging to the diagonal $\{(x,x)\, :\, x\in [0,1]\}$ such that
\[
\Big(I_\eps(f)\Big)(x) = \int_0^1\, K_\eps(x,y)\, f(y)\, dy.
\]
\begin{definition}\label{def:flat-things}
Given a bounded operator $\PP$ on $C^k([0,1])$ set
\[
\PP_\eps(x,y) := \Big(\PP (K_\eps(\cdot,y))\Big)(x) \quad \text{and} \quad  \flattr_\eps(\PP) := \int_0^1\, \PP_\eps(x,x)\, dx.
\]
Then $\PP$ admits a \emph{flat trace} if the following limit exists:
\[
\flattr(\PP) := \lim_{\eps \to 0}\, \flattr_\eps(\PP).
\]
If $\PP^m$ admits a flat trace for all $m\ge 1$, then we call \emph{flat determinant} of $\PP$ the formal power series defined for $u\in \C$ as 
\begin{equation} \label{flat-det}
\flatdet(1-u\,\PP) := \exp\Big( -\sum_{m\ge 1}\, \frac{u^m}m\, \flattr(\PP^m) \Big).
\end{equation}
\end{definition}

We now show that the transfer operators $\QQ_w(z)$ admit a flat determinant and give a relation with $Z_w(z,u)$.

\begin{proposition}\label{prop:q-flat}
Under the assumptions (H1)-(H7), for all $z\in D$ the transfer operator-valued power series $(\QQ_w(z))^m$ admits a flat trace for all $m\ge 1$ and
\[
\flattr((\QQ_w(z))^m) = \sum_{G^m(x)=x}\, \frac{z^{\sum_{j=0}^{m-1} (1+\tau(G^j(x)))} \, \exp\Big(\sum_{j=0}^{m-1}\, w(G^j(x))\Big)}{1-((G^m)'(x))^{-1}}.
\]
It follows that {$z\mapsto \flattr((\QQ_w(z))^m)$ is holomorphic for all $m\ge 1$ on $D^\circ =\{z\in \C: |z|<1\}$ and}
\begin{equation}\label{Z-flatdet}
Z_w(z,u) = \frac{\flatdet(1-u\, \QQ_{w-\log G'}(z))}{\flatdet(1-u\, \QQ_{w}(z))}
\end{equation}
whenever the three terms make sense, where we are using the notation\footnote{By definition of $G$, since $T'_0(x)>0$, the sign of $G'$ is always positive or always negative where it exists. When $G'>0$ the notation is justified by the remark that $\QQ_{w-\log G'}$ is the operator associated to the potential $\tilde w = w-\log G'$.}
\begin{equation} \label{q-w-log}
(\QQ_{w-\log G'}(z)f) (x) = \sum_{n=1}^\infty\, z^n\, \frac{e^{w(\phi_n(x))}}{G'(\phi_n(x))}\,  f(\phi_n(x)).
\end{equation}
\end{proposition}

\begin{proof}
For $z\in D$ and all $m\ge 1$ we can write
\[
((\QQ_w(z))^m f)(x) = \sum_{\beta \in \N^m}\, z^{|\beta|}\, W_\beta(x)\, f(\phi_\beta(x))
\]
where $\beta = (\beta_1,\, \beta_2,\dots,\beta_m)$ and $\phi_\beta :(0,1) \to A_{\beta_1}$ is a strict contraction (see \eqref{notations}) with a unique fixed point $x_\beta \in \phi_\beta(0,1)$. Setting
\[
(Q_\beta f)(x) := W_\beta(x)\, f(\phi_\beta(x)),
\]
let us show that $Q_\beta$ admits a flat trace. For all small $\eps>0$ we have
\[
\flattr_\eps(Q_\beta) = \int_0^1\, (Q_\beta)_\eps(x,x)\, dx = \int_0^1\, \gamma_\eps(\phi_\beta(x)-x)\, W_\beta(x)\, dx,
\]
and using the notation $\Phi_\beta(x)$ for the injective map $(0,1)\ni x\mapsto \phi_\beta(x)-x$, we can write
\[
\flattr_\eps(Q_\beta) = \int_{\Phi_\beta(0,1)}\, \gamma_\eps(z)\, \frac{W_\beta(\Phi_\beta^{-1}(z))}{|1-\phi'_\beta(\Phi_\beta^{-1}(z))|}\, dz\, .
\]
By the properties of $\gamma_\eps$ and using the fixed point $x_\beta = \Phi_\beta^{-1}(0)$, it follows that
\[
\lim_{\eps\to 0}\, \flattr_\eps(Q_\beta) = \frac{W_\beta(x_\beta)}{|1-\phi'_\beta(x_\beta)|} = \flattr(Q_\beta).
\]
In addition, for all $\beta \in \N^m$ we have $|1-\phi'_\beta(x_\beta)| \ge 1-\rho^{-m}$ and
\[
\begin{aligned}
\sum_{\beta \in \N^m}\, W_\beta(x_\beta) = &\, \sum_{\beta_2,..,\beta_m \in \N} \Big(\sum_{\beta_1\in \N}\, \exp (w((\phi_{\beta_1}\circ \phi_{_{\beta_2..\beta_m}})(x_\beta))) \Big) \exp \Big( \sum_{j=2}^m\, w(\phi_{_{\beta_j..\beta_m}}(x_\beta)) \Big) \\
\le &\, c_k\, \sum_{\beta_2,..,\beta_m \in \N}\, \exp \Big( \sum_{j=2}^m\, w(\phi_{_{\beta_j..\beta_m}}(x_\beta)) \Big) \le c_k^m
\end{aligned}
\]
by (H7). Hence we can apply \cite[Lemma 3.21]{bal-book-2} and write for all $m\ge 1$ and all $z\in D$
\[
\flattr((\QQ_w(z))^m) = \sum_{\beta \in \N^m}\, z^{|\beta|}\, \flattr(Q_\beta).
\]
{Hence, $z\mapsto \flattr((\QQ_w(z))^m)$ is holomorphic for all $m\ge 1$ on $D^\circ$. In addition,} rearranging the terms in the sum using the first passage function $\tau(x)$ defined in \eqref{tau-function} proves the first part of the statement. 

Let us now consider the two-variable zeta function $Z_w(z,u)$ where it makes sense. Let us first assume that $G'(x)$ is positive where it is defined. Since
\[
\sum_{G^m(x)=x}\, z^{\sum_{j=0}^{m-1} (1+\tau(G^j(x)))} \, \exp\Big(\sum_{j=0}^{m-1}\, w(G^j(x))\Big) =
\]
\[
= \sum_{G^m(x)=x}\, \frac{z^{\sum_{j=0}^{m-1} (1+\tau(G^j(x)))} \, \exp\Big(\sum_{j=0}^{m-1}\, w(G^j(x))\Big)}{1-((G^m)'(x))^{-1}} \, (1-((G^m)'(x))^{-1}) =
\]
\[
= \flattr((\QQ_w(z))^m) - \sum_{G^m(x)=x}\, \frac{z^{\sum_{j=0}^{m-1} (1+\tau(G^j(x)))} \, \exp\Big(\sum_{j=0}^{m-1}\, w(G^j(x))\Big)\, ((G^m)'(x))^{-1}}{1-((G^m)'(x))^{-1}}\, ,
\]
and
\[
((G^m)'(x))^{-1} = \exp \Big( - \sum_{j=0}^{m-1}\, \log G'(G^j(x)) \Big),
\]
we obtain by \eqref{2var-zeta}
\[
Z_w(z,u) = \exp \Big( \sum_{m\ge 1} \, \frac{u^m}m \, \Big(\flattr((\QQ_w(z))^m) - \flattr((\QQ_{w-\log G'}(z))^m)\Big)\, \Big),
\]
and \eqref{Z-flatdet} follows by the definition of the flat determinant. In the case $G'(x)<0$, the proof follows in the same way.
\end{proof}

At this point, using \eqref{Z-flatdet} and \eqref{rel-zeta-t-fund}, in order to obtain the analytic properties of $\dynzeta{T,v}(z)$ we are reduced to study the function $z\mapsto \flatdet(1-\QQ_{w}(z))$ {with $w$ fixed. Using \eqref{Z-flatdet}, it will be enough to prove that the flat determinants involved are well defined as holomorphic functions of $z$ for $u=1$.}

The jump transformation $G$ is piecewise smooth on $[0,1]$ with full branches and expanding, and the flat determinant for maps with these properties has been studied since the 1990s. The classical results can be found in \cite{ruelle, kitaev} and we refer to the bibliography in \cite[Chapter 3]{bal-book-2} for more recent results and an account of the situation. In our context these results imply that the function $u\mapsto \flatdet(1-u\, \QQ_{w}(1))$ admits a holomorphic extension to the disc\footnote{Recall that $w\in C^k$ with $1\le k\le r-1$.} $\{|u| \le \rho^{-k} \, \exp(-P_G(w))\}$, and the function vanishes at some $u_0$ in this disc if and only if $1/u_0$ is an eigenvalue of $\QQ_{w}(1)$.

Hence, in order to study the function $z\mapsto \flatdet(1-\QQ_{w}(z))$, using the relations in \eqref{2var-zeta-relations} we might look at the transfer operator of $G$ with potential $w+(\log z)(1+\tau)$ and consider the set of $z$ for which $\rho^{-k} \, \exp(-P_G(w+(\log z)(1+\tau)) \ge 1$. A more straightforward approach is to use the method of \cite{LT}.

Let us recall Theorem \ref{thm:q-ck} and the definitions of $\Lambda_m(z)$ in \eqref{lambda-m} and $\Lambda(z)$ in \eqref{lambda-z} for $|z|\le 1$.

\begin{proposition}\label{prop:lt}
Let $z\in D$ and let $\sigma(z) = (1/\rho)^{k-1} \Lambda(z)$ for $k\in \N$ and $k\le r-1$. Then $\flatdet(1-u\, \QQ_{w}(z))$ is holomorphic on the disc $\{|u| < (\sigma(z))^{-\frac 12}\}$ on which
\[
\flatdet(1-u\, \QQ_{w}(z)) = \det(1-u\, R(z))\, \exp \Big( -p_0(u) - \sum_{n=1}^\infty\, a_n\, u^n\Big),
\] 
where $R(z)$ is a finite rank operator, $p_0(u)$ is a polynomial, and $|a_n| = O((\sigma(z))^{\frac n2})$.
\end{proposition}

\begin{proof}
{

For a fixed $z\in D$ for which Theorem \ref{thm:q-ck} holds for $\QQ_w(z)$ on $C^k([0,1])$, following \cite{LT} we first consider the adjoint operator $\QQ_{w}^*(z)$ given by $(\QQ_w(z) f, g)_{L^2} = (f,\QQ_w^*(z)g)_{L^2}$ which reads
\[
(\QQ_w^*(z)g)(y) = z^{1+\tau(y)}\, e^{w(y)}\, |G'(y)|\, g(G(y)).
\]
The Markov structure of $G$ implies that if $g\in C^k_0(0,1)$ then $\QQ_w^*(z)g \in C^k(0,1)$. Hence, we can define the operator
\[
\MM_w(z) := \QQ_{w}(z) \otimes \QQ_{w}^*(z) : C^k(0,1) \otimes C^k_0(0,1) \to C^k(0,1) \otimes C^k(0,1)\, .
\]

Let us now consider $p,q\in \N_0$ with $p,q\le k-1$, and for $f\in C^k([0,1]\times[0,1])$ the \emph{anisotropic norm} (see \cite{GL,DKL})
\[
\| f(x,y) \|_{p,q} := \sup_{x\in (0,1)}\, \sup_{\alpha\le p}\, \sup_{\stackrel{\xi \in C^{q}_0(0,1)}{\|\xi\|_{q}\le1}}\, \left| \int_0^1\, \partial^\alpha_x f(x,y)\, \xi(y)\, dy \right|\, .
\]
Notice that no pointwise regularity in the $y$-variable is used in the definition of $\|\cdot\|_{p,q}$. Indeed, the dependence on $y$ is measured only in duality with test functions $\xi\in C_0^q(0,1)$, so that for $f\in C^k([0,1]\times[0,1])$ the map $y\mapsto f(x,y)$ is merely required to define a distribution of order at least $q$ for each $x$.

Finally, we define the Banach space
\[
\BB^{p,q} := \overline{C^k([0,1]\times[0,1])}^{\| \cdot \|_{p,q}}\, ,
\]
and consider the extension of the operator $\MM_w(z)$ to the space of functions in $C^k([0,1]\times[0,1])$ which have compact support in $y$, given by
\[
(\MM_w(z) f)(x,y)
=
\sum_{n\ge1} z^{n+1+\tau(y)}\, e^{w(\phi_n(x))+w(y)}\, |G'(y)|\, f(\phi_n(x),G(y)),
\]
We have $(\MM_w(z) f)(x,y) \in C^k([0,1]\times[0,1])$. The inequality \eqref{ly-tensor-1} implies that $\MM_w(z)$ can be extended to a continuous linear operator on $\BB^{p,q}$.
}

In the rest of the proof, we show that the functional framework used here fits into the abstract setting
introduced in \cite{LT}. { First, we consider assumptions (P1)-(P3), which involve the action of $\QQ_w(z)$ on $C^p(0,1)$. By Theorem~\ref{thm:q-ck}, for each $z\in D$ the operator $\QQ_w(z)$ acts
boundedly on $C^p(0,1)$ and, for all $p\ge 1$, satisfies a Lasota-Yorke inequality of the form
\[
\|\QQ_w(z)f\|_{p} \le const \cdot\,\Lambda_1(z)\big(\rho^{-p}\|f\|_{p}+\|f\|_{p-1}\big)\, .
\]
Hence, (P1)-(P3) follow. Furthermore, (P4) follows by the definition of the space $\BB^{p,q}$, which plays the role of the space $\tilde \BB$ used in \cite{LT}, and (P5) follows from the standard embedding of $\BB^{p,q}$ into the set of distributions on $[0,1]\times[0,1]$.

We can then check (P8).} Let $\delta$ be the distribution on $[0,1]\times[0,1]$ defined by
\[
\langle \delta,\varphi\rangle = \int_0^1 \varphi(x,x)\,dx,
\qquad \varphi\in C^k([0,1]\times[0,1]).
\]
{Its extension to a continuous linear functional on $\BB^{p,q}$ follows since $\| \varphi\|_\infty \le \|\varphi \|_{p,q}$ for all $\varphi\in C^k([0,1]\times[0,1])$. (P7) follows from the extension of $\MM_w(z)$ to $\BB^{p,q}$.

Finally, (P9) and (P10) follow from the standard approximation of kernels on $C^k(0,1)$.

Then, it remains to prove (P6) of \cite{LT}.  We show that for all 
$f\in C^k([0,1]\times[0,1])$
\begin{align}
& \| (\MM_w(z))^m f \|_{p,q} \le const \cdot\, \Lambda_m(z)\, \| f \|_{p,q}  \label{ly-tensor-1} \\[0.2cm]
& \| (\MM_w(z))^m f \|_{p,q} \le const \cdot\, \Lambda_m(z)\, \Big( \rho^{-\min\{p,q\}\, m}\, \| f \|_{p,q} + \| f \|_{p-1,q+1}\Big), \label{ly-tensor-2}
\end{align}
where the constants do not depend on $m$ and $f$. This implies that the operator $\MM_w(z)$ extends uniquely by continuity to a bounded operator on $\BB^{p,q}$ and (P6).

At this point the proposition follows by choosing $p=q=k-1$ and repeating the arguments in \cite[Section 4]{LT} verbatim.} 

The rest of the proof is dedicated to show \eqref{ly-tensor-1} and \eqref{ly-tensor-2}. Let's begin with the case $m=1$. For $p=0$ we have
\[
\begin{split}
&\| \MM_w(z) f \|_{0,q} \\
&= \sup_{x\in (0,1)}\, \sup_{\stackrel{\xi \in C^{q}_0(0,1)}{\|\xi\|_{q}\le1}}\, \left|\sum_{n\ge 1}\, z^n\, e^{w(\phi_n(x))}\, \int_0^1\, z^{1+\tau(y)}\, e^{w(y)}\, |G'(y)|\, f(\phi_n(x),G(y))\, \xi(y)\, dy \right|
\end{split}
\]
and since
\[
\begin{split}
&\int_0^1\, z^{1+\tau(y)}\, e^{w(y)}\, |G'(y)|\, f(\phi_n(x),G(y))\, \xi(y)\, dy= \\
 &=\sum_{\ell \ge 1}\, z^{\ell}\, \int_{A_\ell}\, e^{w(y)}\, |G'(y)|\, f(\phi_n(x),G(y))\, \xi(y)\, dy= \\
&= \sum_{\ell \ge 1}\, z^{\ell}\, \int_0^1\, e^{w(\phi_\ell(s))}\, f(\phi_n(x),s)\, \xi(\phi_\ell(s))\, ds= \\
&= \int_0^1\, f(\phi_n(x),s)\, (\QQ_w(z) \xi)(s)\, ds,
\end{split}
\] 
and by \eqref{lambda-m}
\[
\sum_{n\ge 1}\, |z|^n\, e^{w(\phi_n(x))} \le \Lambda_1(z),
\]
we use that $\| \QQ_w(z) \xi\|_q \le const(q)\, \|\xi\|_q$ to obtain \eqref{ly-tensor-1} for $m=1$ and $p=0$.

For $p=1$ we need to control the term involving $\partial_x (\MM_w(z)f)$. Since
\[
\begin{aligned}
\partial_x (\MM_w(z)f)(x,y) = \sum_{n\ge 1}\, & z^{n+1+\tau(y)}\, e^{w(\phi_n(x))+w(y)}\, |G'(y)|\, \phi'_n(x)\, \cdot \\[0.2cm]
&\cdot \, \Big( w'(\phi_n(x))\, f(\phi_n(x),G(y)) + \partial_x f (\phi_n(x),G(y))\Big),
\end{aligned}
\]
we have
\begin{equation}\label{1q-norms}
\Big| \int_0^1\, \partial_x (\MM_w(z)f)(x,y)\, \xi(y)\, dy\Big| \le const(w)\, \|\MM_w(z)f\|_{0,q} + \rho^{-1}\, \|\MM_w(z) (\partial_x f)\|_{0,q}.
\end{equation}
Then we use \eqref{ly-tensor-1} and the definition of $\|\cdot\|_{1,q}$ to write
\[
\|\MM_w(z) (\partial_x f)\|_{0,q} \le const\cdot\, \Lambda_1(z)\, \|f\|_{1,q},
\]
and we show that
\begin{equation}\label{1q-norms-step}
\|\MM_w(z)f\|_{0,q} \le const\cdot  \Lambda_1(z)\, \Big(\rho^{-q}\, \|f\|_{1,q} + \rho^{q}\, \|f\|_{0,q+1}\Big),
\end{equation}
to obtain in \eqref{1q-norms}
\[
\Big| \int_0^1\, \partial_x (\MM_w(z)f)(x,y)\, \xi(y)\, dy\Big| \le const\cdot  \Lambda_1(z)\, \Big(\rho^{-1}\,  \|f\|_{1,q} + \rho^{-q}\, \|f\|_{1,q} + \rho^{q}\, \|f\|_{0,q+1}\Big)
\]
which gives \eqref{ly-tensor-2} for $m=1$ and $p=1$. It remains to prove \eqref{1q-norms-step}. Consider a fixed $\gamma:(0,1)\to [0,\infty)$ of class $C_0^{q+1}$ with $\int_0^1\, \gamma(s) \, ds =1$. For $\eps>0$ small let $\gamma_\eps(s) := \eps^{-1} \gamma(s/\eps)$, then we denote $\xi_\eps(y) := (\xi \star \gamma_\eps) (y) \in C^{q+1}$ for which we have\footnote{The claim is straightforward since 
\[
\int \gamma_\varepsilon (x-y) \xi(y) dy= \int \xi(x-y)\gamma_\varepsilon (y) dy.
\]}
\begin{equation} \label{mollifiers}
\| \xi_\eps \|_{q+1} \le const\cdot \eps^{-1}, \quad \| \xi-\xi_\eps\|_{q-1}\le \eps\, \|\xi\|_q, \quad \| \xi-\xi_\eps\|_{q}\le const,
\end{equation}
where the constants do not depend on $\eps$. We write
\[
\begin{split}
&\int_0^1\, (\MM_w(z) f)(x,y)\, \xi(y)\, dy= \\
&= \int_0^1\, (\MM_w(z) f)(x,y)\, \xi_\eps(y)\, dy + \int_0^1\, (\MM_w(z) f)(x,y)\, (\xi(y)-\xi_\eps(y))\, dy
\end{split}
\]
and as above
\[
\begin{split}
&\int_0^1\, (\MM_w(z) f)(x,y)\, (\xi-\xi_\eps)(y)\, dy= \\
&= \sum_{n,\ell \ge 1}\, z^{n+\ell}\, e^{w(\phi_n(x))}\, \int_0^1\, e^{w(\phi_\ell(s))}\, f(\phi_n(x),s)\, (\xi-\xi_\eps)(\phi_\ell(s))\, ds.
\end{split}
\]
Using \eqref{mollifiers}, we have for all $\ell\ge 1$
\[
\begin{split}
\| (\xi-\xi_\eps)(\phi_\ell(s)) \|_q \le &\, \| (\xi-\xi_\eps)^{(q)}(\phi_\ell(s)) \cdot (\phi_\ell(s))^q \|_\infty + const\, \| (\xi-\xi_\eps)(\phi_\ell(s)) \|_{q-1} \le \\[0.2cm]
\le &\, const\, (\rho^{-q} + \eps).
\end{split}
\]
And finally
\[
\|\MM_w(z)f\|_{0,q} \le const\, \Big( \eps^{-1} \, \|\MM_w(z)f\|_{0,q+1} + (\rho^{-q} + \eps)\, \|\MM_w(z)f\|_{0,q} \Big).
\]
Choosing $\eps = \rho^{-q}$, using \eqref{ly-tensor-1} and the definition of $\|\cdot\|_{1,q}$ we obtain \eqref{1q-norms-step}.

{We now give some details on how to prove \eqref{ly-tensor-1}-\eqref{ly-tensor-2} for $p>1$. The argument is similar to that in the proof of Theorem \ref{thm:q-ck}. We treat first $m=1$ and then consider arbitrary $m>1$ by iteration.

\noindent
\emph{Step 1: structure of $\partial_x^p(\MM_w(z)f)$.}
Recall that
\[
(\MM_w(z) f)(x,y)=\sum_{n\ge 1} z^{n+1+\tau(y)} e^{w(\phi_n(x))+w(y)}|G'(y)|\, f(\phi_n(x),G(y)).
\]
Fix $p\in\{2,\dots,k-1\}$. Using repeatedly Leibniz' rule and the Fa\`a di Bruno's formula for the composition $x\mapsto f(\phi_n(x),G(y))$, we can write
\begin{equation}\label{eq:Mp-decomp}
\begin{aligned}
\partial_x^p(\MM_w(z)f)(x,y)
= \sum_{n\ge 1} & z^{n+1+\tau(y)} e^{w(\phi_n(x))+w(y)}|G'(y)| \cdot \\[0.3cm]
& \cdot \Big((\phi_n'(x))^p\, \partial_x^p f(\phi_n(x),G(y)) + \mathcal{E}_{n,p}(f)(x,y)\Big),
\end{aligned}
\end{equation}
where $\mathcal{E}_{n,p}(f)$ is a finite sum of terms of the following type:
\[
\mathcal{T}_{n,p}(f)(x,y)
:=
A_{n,p}(x)\, \partial_x^{j} f(\phi_n(x),G(y)),
\qquad 0\le j\le p-1.
\]
Here $A_{n,p}(x)$ is a $C^{k-p}$ function depending only on $w$ and on $\phi_n$ and their derivatives up to order $p$. By the uniform expanding property of $G$ and the standard distortion bounds for inverse branches, there exist positive constants such that
\[
\|A_{n,p}\|_\infty \le const(w,p)\, \|\phi_n\|_p^{p}\le const(w,p)\, , \quad \|\phi_n'\|_\infty^p\le \rho^{-p},
\]
uniformly in $n\ge 1$.

\noindent
\emph{Step 2: estimate of the main term.} 
Let $\xi\in C_0^q(0,1)$ with $\|\xi\|_q\le 1$.
For the first term in \eqref{eq:Mp-decomp}, we apply the change of variables on each interval $A_\ell$ as in the case $p=0$, to write
\[
\sum_{n\ge 1}\, z^{n}\, e^{w(\phi_n(x))}\, (\phi_n'(x))^p\, \sum_{\ell \ge 1}\, z^\ell\,  \int_{A_\ell} e^{w(y)}\, |G'(y)|\, \partial_x^p f(\phi_n(x),G(y))\, \xi(y)\, dy =
\]
\[
=\sum_{n\ge 1}\, z^{n+1}\, e^{w(\phi_n(x))}\, (\phi_n'(x))^p\, \int_0^1 \partial_x^p f(\phi_n(x),s)\, (\QQ_w(z)\xi)(s)\, ds\, .
\]
Therefore, arguing as above,
\begin{equation}\label{eq:main-term-p}
\begin{split}
\sup_{x\in (0,1)}\, \sup_{\stackrel{\xi \in C^{q}_0(0,1)}{\|\xi\|_{q}\le1}}\,
\Big|\sum_{n\ge 1} z^n e^{w(\phi_n(x))} (\phi_n'(x))^p &
\int_0^1 z^{1+\tau(y)} e^{w(y)}|G'(y)|\, \partial_x^p f(\phi_n(x),G(y))\,\xi(y)\,dy\Big|\\
&\le const \cdot\Lambda_1(z)\,\rho^{-p}\,\|f\|_{p,q}.
\end{split}
\end{equation}

\noindent
\emph{Step 3: estimate of the error term via the $(p-1,q+1)$ norm.}
We now consider the contribution of $\mathcal{E}_{n,p}(f)$ in \eqref{eq:Mp-decomp}.
Each $\mathcal{T}_{n,p}(f)$ involves some derivatives $\partial_x^j f(\phi_n(x),G(y))$ with $0\le j\le p-1$ multiplied by a bounded coefficient $A_{n,p}(x)$. Arguing as above, we need to estimate terms of the form
\[
 \|A_{n,p}\|_\infty\, \sup_{s\in(0,1)}\Big|\int_0^1 \partial_x^j f(\phi_n(x),s)\, (\QQ_w(z)\xi)(s)\, ds\Big|,
\]
with $\xi \in C^q_0(0,1)$. Since $\QQ_w(z)\xi \in C^q$ with $\|\QQ_w(z)\xi \|_q\le const$, we can apply the same mollification argument used to prove \eqref{1q-norms-step} (with $j$ in place of $0$), gaining one derivative in the $y$-test function.
More precisely, one obtains
\begin{equation}\label{eq:error-p}
\begin{split}
\sup_{x\in (0,1)}\, \sup_{\stackrel{\xi \in C^{q}_0(0,1)}{\|\xi\|_{q}\le1}}\,
& \Big|\sum_{n\ge 1} z^n e^{w(\phi_n(x))}\int_0^1 z^{1+\tau(y)}e^{w(y)}|G'(y)|\, \mathcal{E}_{n,p}(f)(x,y)\,\xi(y)\,dy\Big|\\
&\le const(w,p)\cdot \Lambda_1(z)\,\|f\|_{p-1,q+1}.
\end{split}
\end{equation}

\smallskip\noindent
\emph{Step 4: conclusion.}
Combining \eqref{eq:main-term-p} and \eqref{eq:error-p} yields, for all $p\in\{2,\dots,k-1\}$ and $q\le k-1$,
\[
\|\MM_w(z) f\|_{p,q}\le C\,\Lambda_1(z)\,\big(\rho^{-p}\|f\|_{p,q}+\|f\|_{p-1,q+1}\big),
\]
which is \eqref{ly-tensor-2} for $m=1$. The bound \eqref{ly-tensor-1} follows similarly by estimating all terms directly by $\Lambda_1(z)\|f\|_{p,q}$.
The estimates for $(\MM_w(z))^m$ follow by iterating the inequalities of the case $m=1$ and by using
$\Lambda_{m+m'}(z)\le \Lambda_m(z)\Lambda_{m'}(z)$ (see Theorem~\ref{thm:q-ck}-(i)) to control the growth of the weight. Thus, \eqref{ly-tensor-1} and \eqref{ly-tensor-2} are proved for all $m\ge 1$.}
\end{proof}

By Propositions \ref{prop:standard}, \ref{prop:q-flat}, and \ref{prop:lt},  we obtain the proof of Theorem \ref{thm:zeta-final-1} on the meromorphic extension for the zeta function $\dynzeta{T,v}(z)$. \\

\noindent \textbf{Proof of Theorem \ref{thm:zeta-final-1}}.
{We put together Propositions \ref{prop:standard} and \ref{prop:q-flat}, and apply Proposition \ref{prop:lt} to $\QQ_w(z)$ and $\QQ_{w-\log G'}(z)$. 

By Theorem \ref{thm:q-ck} and Proposition \ref{prop:q-flat}, the functions $z\mapsto \QQ_w(z)$ and $z\mapsto \flattr((\QQ_w(z))^m)$, for all $m\ge 1$, are holomorphic on $D^\circ$. The same holds for $\QQ_{w-\log G'}(z)$ and its flat traces by comparison.

In Proposition \ref{prop:q-flat}, we also obtain the estimate $|\flattr((\QQ_w(z))^m)| \le c_k^m$ for all $m\ge 1$, therefore, using \eqref{flat-det}, we obtain that $\flatdet(1-u\, \QQ_{w}(z))$ is a holomorphic function on the polydisc $\{(u,z)\in \C^2\, :\, |u|< c_k^{-1}\, ,\, |z|<1\}$. Moreover, Proposition \ref{prop:lt} implies that $u\mapsto \flatdet(1-u\, \QQ_{w}(z))$ can be extended to the disc $\{|u|<\sigma(z)\}$ for all $z\in D^\circ$. Since $\sigma(z) \ge \rho\cdot r_{ess}(\QQ_w(z)|_{C^{k}})$ for all $k<r$ by Theorem \ref{thm:q-ck}, choosing $k=r-1$ we get that $\Lambda(z)<\rho^{r-2}$ implies $\sigma(z)>1$. Hence, an application of Hartogs' extension Theorem for cylinders with variable radius (see e.g. \cite{ho}) implies that $z\mapsto \flatdet(1-\QQ_{w}(z))$ is holomorphic on the set $z\in D^\circ$ such that $\Lambda(z)<\rho^{r-2}$. 

Finally, arguing as in Theorem \ref{thm:q-ck} for $\QQ_{w-\log G'}(z)$, it is immediate that Lemma \ref{lem:prop-jt} implies that the limit $\tilde \Lambda(z)$ analogous to the $\Lambda(z)$ for $\QQ_w(z)$, satisfies $\tilde \Lambda(z)\le \Lambda(z)$. Hence, $\Lambda(z)<\rho^{r-2}$ implies $\tilde\Lambda(z)<\rho^{r-2}$. Moreover, $\flattr((\QQ_{w-\log G'}(z))^m)$ is bounded by $c_k^m$ for all $m\ge 1$. Therefore, also $z\mapsto \flatdet(1-\QQ_{w-\log G'}(z))$ is holomorphic on the set $z\in D^\circ$ such that $\Lambda(z)<\rho^{r-2}$.
}
\qed

\section{The meromorphic extension outside the unit disc} \label{sec:meromorphic}
{The result in Theorem \ref{thm:zeta-final-1} gives the meromorphic extension of $\dynzeta{T,v}(z)$ on the set $D^\circ$, for which $\QQ_w(z)$ and its flat determinant are proved to be holomorphic. We now show a strategy to extend the function $z\mapsto \QQ_w(z)$ as an analytic function to larger domains. We believe this implies the meromorphic extension of the dynamical zeta function $\dynzeta{T,v}(z)$ to a domain larger than that of Theorem \ref{thm:zeta-final-1} (see Remark \ref{rem:zeta-ext}).}

We recall the following result by LeRoy and Lindel\"of as stated in \cite{ara}. Let us set
\[
\begin{aligned}
& \Delta_c := \set{ z\in \C\, :\, |\arg z| < \frac c2}, \quad \text{for $c\in [0,2\pi)$;} \\[0.2cm]
& \Pi := \set{ z\in \C\, :\, \real{z} \ge 0}.
\end{aligned}
\]
Let us denote by $H(\Pi)$ the set of holomorphic functions in some neighborhood of $\Pi$. Then we say that $h\in H(\Pi)$ is \emph{of exponential type $e_{_\Pi} \in [0,+\infty)$ on $\Pi$} if
\[
e_{_\Pi} = \limsup_{z\to \infty,\, z\in \Pi}\, |z|^{-1}\, \log^+ |h(z)|.
\]
Then the following result holds for the analytic continuation of power series.

\begin{theorem} \label{teo:fund-ps}
For a function $h\in H(\Pi)$, assume that the power series 
\[
q(z) = \sum_{n=0}^\infty\, h(n)\, z^n
\]
converges in the unit disc $D$. If $h$ is of exponential type $e_{_\Pi} < \pi$ on $\Pi$ then $q(z)$ admits an analytic continuation to the sector $\C \setminus \Delta_{2e_{_\Pi}}$.
\end{theorem}

In the case $e_{_\Pi} =0$ the theorem yields a sufficient condition for the analytic continuation of $q(z)$ to the cut complex plane $\C \setminus (1,\infty)$.

Theorem \ref{teo:fund-ps} has been improved in \cite{ara} by stating a necessary and sufficient condition on a function $h\in H(\Pi)$ for the analytic continuation of the power series $q(z)$ to a sector $\C \setminus \Delta_{c}$. However, for our aims it is enough to consider the sufficient condition considered by LeRoy and Lindel\"of. It is shown in \cite{ara} that the proof of Theorem \ref{teo:fund-ps} follows by choosing for each compact set $K\subset \C \setminus \Delta_{2e_{_\Pi}}$ an unbounded contour $\Gamma$ on which the integral
\[
\int_\Gamma\, \frac{h(s)\, z^s}{e^{2\pi i s}-1}\, ds
\]
converges uniformly to $q(z)$ for all $z\in K$.

Let $f\in C^k([0,1])$ and consider the series 
\[
(\QQ_w(z)f) (x) = \sum_{n=1}^\infty\, z^n\, e^{w(\phi_n(x))}\, f(\phi_n(x))
\]
which, by assumption (H7), converges for $z\in D$ uniformly in $x$. Fix $x\in [0,1]$, we study the existence of a function $h\in H(\Pi)$ such that $h(n) = e^{w(\phi_n(x))}\, f(\phi_n(x))$ for all $n\ge 1$.

\begin{lemma} \label{lem:extensions}
Let $a_n = e^{w(\phi_n(x))}\, f(\phi_n(x))$ and assume that $v(0)<0$. Then for all $\eps>0$, the function
\[
h(s) := \frac{\sin (\pi s)}{\pi}\, \sum_{i=1}^\infty\, (-1)^i\, \frac{a_i}{s-i}\, e^{(1-\eps)\, (s-i)\, v(0)}
\]
converges in $\C$ and satisfies $h(n)=a_n$ for all $n\ge 1$. Moreover it is of exponential type $e_{_\Pi} = \max\{0, \pi+(1-\eps) v(0)\} < \pi$.
\end{lemma}

\begin{proof}
Given the sequence $\{a_n\}$, an application of the Pringsheim interpolation formula gives the function $h$ we are looking for. In details, the function $g(s)= \sin (\pi s)$ has zeroes of first order at all $n\in \Z$, therefore the formal power series
\[
\tilde h(s):=\sum_{i=1}^\infty\, \frac{g(s)}{g'(i)\, (s-i)}
\]
satisfies $\tilde h(n) =1$ for all $n\in \N$. Since $g'(i) = (-1)^i\, \pi$ and $a_n \sim e^{n\, v(0)}\, f(0)$ as showed in Remark \ref{rem:v0}, the series defining $h(s)$ is convergent in $\C$ and satisfies $h(n) = a_n$ for all $n\ge 1$.

In addition, for all $s\in \Pi$ we have
\[
|h(s)| = |\sin (\pi s)|\, e^{(1-\eps)\, v(0)\, \Re(s)}\, \sum_{i=1}^\infty\, \frac{|a_i|}{\pi\, |s-i|}\, e^{-(1-\eps)\, v(0)\, i}.
\]
Since the series converges if $\eps>0$ and $|\sin (\pi s)| \le 2\, e^{\pi\, |s|}$, we obtain that $e_{_\Pi} = \max\{0, \pi+(1-\eps) v(0)\} $.
\end{proof}

\noindent \textbf{Proof of Theorem \ref{teo:cont-v0}}. Apply Lemma \ref{lem:extensions} for all $\eps>0$. \qed

\subsection*{Acknowledgements} This work was started when RC was working at University of Pisa, Italy.\\ CB acknowledges the MIUR Excellence Department Project awarded to the Department of Mathematics, University of Pisa, CUP I57G22000700001.\\ The authors are partially supported by the research project PRIN 2022NTKXCX ``Stochastic properties of dynamical systems'' funded by the Ministry of University and Scientific Research of Italy.\\ This research is part of the authors' activity within the the UMI Group ``DinAmicI'' \texttt{www.dinamici.org} and the INdAM (Istituto Nazionale di Alta Matematica) group GNFM.\\ It is a pleasure to thank Paolo Giulietti for many useful discussions at the beginning of the project, {and the referees for very useful comments and suggestions which have largely improved the exposition of the paper.}

\subsection*{Declarations}
The authors have no relevant financial or non-financial interests to disclose. Data sharing not applicable to this article as no datasets were generated or analysed during the current study.



\begin{thebibliography}{99}

\bibitem{ara} N.~U.~Arakelyan, \emph{On efficient analytic continuation of power series}. Math. USSR-Sb. \textbf{52} (1985), no. 1, 21--39.

\bibitem{bal-book-1} V.~Baladi, ``Positive transfer operators and decay of correlations''. Adv. Ser. Nonlinear Dynam., 16 World Scientific Publishing Co., Inc., River Edge, NJ, 2000.

\bibitem{bal-book-2} V.~Baladi, ``Dynamical zeta functions and dynamical determinants for hyperbolic maps''. Ergeb. Math. Grenzgeb. (3), 68, Springer, Cham, 2018.

\bibitem{BaCa} 	V.~Baladi, R.~Castorrini. \emph{Thermodynamic formalism for piecewise expanding maps in finite dimension}. Discrete Contin. Dyn. Syst., doi: 10.3934/dcds.2024023

\bibitem{BI} C.~Bonanno, S.~Isola, \emph{A thermodynamic approach to two-variable Ruelle and Selberg zeta functions via the Farey map}. Nonlinearity {\bf 27} (2014), no. 5, 897--926.

\bibitem{BL} C.~Bonanno, M.~Lenci, \emph{Pomeau-Manneville maps are global-local mixing}. Discrete Contin. Dyn. Syst. \textbf{41} (2021), no. 3, 1051--1069.

\bibitem{BuCaCa} O.~Butterley, G.~Canestrari, R.~Castorrini, \emph{Discontinuities cause essential spectrum on surfaces}. Annales Henri Poincar\'e.  Volume 26, pages 3075–3101, (2025)


\bibitem{BuCaJa} O.~Butterley, G.~Canestrari, S.~Jain, \emph{Discontinuities cause essential spectrum}. Commun. Math. Phys. \textbf{398} (2023), no. 2, 627--653.

\bibitem{CI} P.~Collet, S.~Isola, \emph{On the essential spectrum of the transfer operator for expanding Markov maps}. Commun. Math. Phys. \textbf{139} (1991), no. 3, 551--557.

\bibitem{DKL}  M.~F.~Demers, N.~Kiamari, C.~Liverani, ``Transfer operators in hyperbolic dynamics. An introduction''. $33^o$  Col\'oq. Bras. Mat. Instituto Nacional de Matem\'atica Pura e Aplicada (IMPA), Rio de Janeiro, 2021.

\bibitem{GL} S.~Gou\"ezel, C.~Liverani, \emph{Banach spaces adapted to Anosov systems}. Ergodic Theory Dynam. Systems \textbf{26} (2006), no. 1, 189--217.

\bibitem{Faa}  L.~Hern\'andez Encinas, J.~Masqu\'e, \emph{A short proof of the generalized Fa\`a di Bruno's formula}. Appl. Math. Lett. \textbf{16} (2003), no. 6, 975--979.

\bibitem{ho} L.~H\"ormander, ``An introduction to complex analysis in several variables''. North-Holland Publishing Co., Amsterdam, London, 1973.

\bibitem{Is0} S.~Isola, \emph{Dynamical zeta functions and correlation functions for non-uniformly hyperbolic transformations}. Preprint available at https://pdodds.w3.uvm.edu/teaching/courses/2009-08UVM-300/docs/others/1995/isola1995a.pdf

\bibitem{Is1} S.~Isola, \emph{On the spectrum of Farey and Gauss maps}. Nonlinearity {\bf 15} (2002), no. 5, 1521--1539.

\bibitem{Is2} S.~Isola, \emph{On systems with finite ergodic degree}. Far East J. Dyn. Syst. {\bf 5} (2003), no. 1, 1--62.

\bibitem{kalle} C.~Kalle, V.~Matache, M.~Tsujii, E.~Verbitskiy, \emph{Invariant densities for random continued fractions}. J. Math. Anal. Appl. {\bf 512} (2022), no. 2, Paper No. 126163.

\bibitem{kitaev}  A.~Yu.~Kitaev, \emph{Fredholm determinants for hyperbolic diffeomorphisms of finite smoothness}. Nonlinearity \textbf{12} (1999), no. 1, 141--179. (Corrigendum. Nonlinearity \textbf{12} (1999), no. 6, 1717--1719.)

\bibitem{LSV} C.~Liverani, B.~Saussol, S.~Vaienti, \emph{A probabilistic approach to intermittency}. Ergodic Theory Dynam. Systems \textbf{19} (1999), no. 3, 671--685.

\bibitem{LT} C.~Liverani, M.~Tsujii, \emph{Zeta functions and dynamical systems}. Nonlinearity \textbf{19} (2006), no.10, 2467--2473.

\bibitem{M} P.~Manneville, \emph{Intermittency in dissipative dynamical systems}. Phys. Lett. A \textbf{79} (1980), no. 1, 33--35.

\bibitem{PY} M.~Pollicott, M.~Yuri, \emph{Zeta functions for certain multi-dimensional non-hyperbolic maps}. Nonlinearity \textbf{14} (2001), no. 5, 1265--1278.

\bibitem{PM} Y.~Pomeau, P.~Manneville, \emph{Intermittent transition to turbulence in dissipative dynamical systems}. Comm. Math. Phys. \textbf{74} (1980), no. 2, 189--197.


\bibitem{ruelle} D.~Ruelle, \emph{An extension of the theory of Fredholm determinants}. Inst. Hautes \'Etudes Sci. Publ. Math. \textbf{72} (1990), 175--193.

\bibitem{Ru} H.~H.~Rugh, \emph{Intermittency and regularized Fredholm determinants}. Invent. Math. {\bf 135} (1999), no. 1, 1--24.

\end{thebibliography}
\end{document}